\documentclass[12pt]{article}

\usepackage{amsmath}
\usepackage{enumitem}
\usepackage{hyperref}
\usepackage[amsmath,thmmarks,hyperref]{ntheorem}

\usepackage{txfonts}

\theoremnumbering{arabic}
\theoremstyle{plain}
\theoremseparator{.}
\RequirePackage{latexsym}
\theoremsymbol{}
\theorembodyfont{\itshape}
\theoremheaderfont{\normalfont\bfseries}
\newtheorem{theorem}{Theorem}[section]

\newtheorem{lemma}[theorem]{Lemma}

\newtheorem{conjecture}[theorem]{Conjecture}
\newtheorem{question}[theorem]{Question}

\theoremheaderfont{\upshape}
\newtheorem{claim}{Claim}[theorem]

\theorembodyfont{\upshape}
\theoremsymbol{}
\theoremheaderfont{\normalfont\bfseries}
\newtheorem{definition}[theorem]{Definition}
\newtheorem{remark}[theorem]{Remark}
\newtheorem{fact}[theorem]{Fact}

\theoremstyle{nonumberplain}
\theoremheaderfont{\itshape}
\theorembodyfont{\normalfont}
\theoremsymbol{\ensuremath{\square}}
\RequirePackage{amssymb}
\newtheorem{proof}{Proof}

\theoremsymbol{\ensuremath{\dashv}}

\qedsymbol{\ensuremath{\square}}
\theoremclass{LaTeX}

\newcommand{\Ord}{\ensuremath{\text{{\rm Ord}}}}

\newcommand{\ZFC}{\ensuremath{\text{{\sf ZFC}}}}
\newcommand{\ZF}{\ensuremath{\text{{\sf ZF}}}}

\newcommand{\Con}{\ensuremath{\text{{\it Con}}}}

\newcommand{\FA}{\ensuremath{\text{{\sf FA}}}}
\newcommand{\MA}{\ensuremath{\text{{\sf MA}}}}
\newcommand{\PFA}{\ensuremath{\text{{\sf PFA}}}}
\newcommand{\BPFA}{\ensuremath{\text{{\sf BPFA}}}}
\newcommand{\BFA}{\ensuremath{\text{{\sf BFA}}}}
\newcommand{\BMM}{\ensuremath{\text{{\sf BMM}}}}

\newcommand{\MM}{\ensuremath{\text{{\sf MM}}}}

\newcommand{\SSP}{\ensuremath{\text{{\sf SSP}}}}

\let\mathbb=\varmathbb

\title{Martin's maximum revisited}
\author{Matteo Viale}
\date{}
\begin{document}
\maketitle
\begin{abstract}
We present several results relating the general theory of the stationary tower forcing developed by
Woodin with forcing axioms. 
In particular we show that, in combination with strong large cardinals,
the forcing axiom $\MM^{++}$ makes the $\Pi_2$-fragment of
 the theory of $H_{\aleph_2}$ 
invariant with respect to stationary set preserving forcings that preserve $\BMM$. 
We argue that this is a close to optimal generalization to
$H_{\aleph_2}$ of Woodin's absoluteness results for 
$L(\mathbb{R})$.
In due course of proving this we shall give a new proof of some of 
Woodin's results.
\end{abstract}




\section{Introduction}
In this introduction we shall take a long detour to motivate the results we want to present and to 
show how they stem out of Woodin's work on 
$\Omega$-logic.
We tried to make this introduction comprehensible 
to any person acquainted with the theory of forcing 
as presented for example in~\cite{kunen}.
The reader may refer to 
subsection~\ref{subsec:not} for unexplained notions.

Since its discovery in the early sixties by Paul Cohen~\cite{COH63}, 
forcing has played a central role in the development of modern set theory.
It was soon realized its fundamental role to establish the undecidability in $\ZFC$
of all the classical problems of set theory, among which Cantor's continuum problem.
Moreover, up to date, forcing (or class forcing) 
is the unique efficient method to obtain independence results over $\ZFC$. 
This method has found applications in virtually all fields of pure mathematics: 
in the last forty years natural problems of group theory, functional analysis, 
operator algebras, general topology, and many other subjects were shown to 
be undecidable by means of forcing (see~\cite{FAR11,SHE74} among others).
Perhaps driven by these observations Woodin introduced $\Omega$-logic, a non-constructive 
semantics for $\ZFC$ which rules out the independence results obtained by means of forcing.
 \begin{definition}
Given a model $V$ of $\ZFC$ and a family 
$\Gamma$ of partial orders in $V$, we say that
$V$ models that $\phi$ is $\Gamma$-consistent if $V^{\mathbb{B}}\models \phi$
for some $\mathbb{B}\in \Gamma$.\end{definition}
The notions of $\Gamma$-validity and of
$\Gamma$-logical consequence $\models_{\Gamma}$ are defined accordingly. 
Woodin's $\Omega$-logic is the $\Gamma$-logic obtained by letting $\Gamma$ be the class of 
all partial orders\footnote{There is a slight twist between Woodin's original definition of 
$\Omega$-consistency and our definition of $\Gamma$-consistency
when $\Gamma$ is the class of all posets. We shall explain in this footnote
why we decided to modify Woodin's original definition. On a first reading the reader may
skip it over. 
Woodin states that $\phi$ is $\Omega$-consistent in $V$ if there is some $\alpha$ and
some $\mathbb{B}\in V_\alpha$ such that $V^{\mathbb{B}}_\alpha\models\phi$.
The advantage of our definition (with respect to Woodin's)
is that it allows for
a simpler formulation of the forcing absoluteness 
results which are the motivation and the purpose of this paper and which assert
that over any model $V$ of some theory $T$ which extends $\ZFC$ any
statement $\phi$ of a certain form which $V$ models to be $\Gamma$-consistent actually holds in $V$. 
To appreciate the difference between Woodin's definition of $\Omega$-consistency and the current definition, 
assume that $\phi$ is a $\Pi_2$-formula and
that $\phi$ is $\Omega$-consistent in $V$ in the sense of Woodin: 
this means that there exist $\alpha$ and $\mathbb{B}$ such that 
$V_\alpha^{\mathbb{B}}\models\phi$, nonetheless it is well
 possible that $V^{\mathbb{B}}\not\models\phi$ and thus that $\mathbb{B}$ does
 not witness that $\phi$ is $\Omega$-consistent according to our definition. 
 Now if $V$ models $\ZFC+$\emph{there are class many Woodin cardinals which are a limit of Woodin
 cardinals} and $\phi^{L(\mathbb{R})}$ is 
 $\Omega$-consistent in $V$ in the sense of Woodin, this can be reflected 
 in the assertion that $\exists\alpha\in V$, $V_\alpha\models \phi^{L(\mathbb{R})}$, but not in the statement
 that $\phi^{L(\mathbb{R})}$ holds in $V$.
 On the other hand if $V$ models 
 $\ZFC+$\emph{there are class many Woodin cardinals which are a limit of Woodin
 cardinals} and $\phi^{L(\mathbb{R})}$ is $\Omega$-consistent in $V$
 according to our definition, we can actually reflect this fact in the assertion 
 that $V\models \phi^{L(\mathbb{R})}$.
 There is no real discrepancy on the two definitions because for each $n$ we can find
 a $\Sigma_n$ formula $\phi_n$ such that if $V$ is any model of $\ZF$,
 $V_\alpha\models\phi_n$ if and only if $V_\alpha\prec_{\Sigma_n} V$.
 Thus, if we want to prove that a certain $\Sigma_n$-formula 
 $\phi$ is $\Omega$-consistent according
 to our definition,
 we just have to prove that $\phi_n\wedge\phi$ is 
 $\Omega$-consistent in $V$ according to Woodin's definition.
 On the other hand the set of $\Gamma$-valid statements (according to Woodin's
 definition) is definable in $V$ in the parameters used to define $\Gamma$, while
 (unless we subsume that there is some $\delta$ such that $V_\delta\prec V$ and all the parameters used to define $\Gamma$ belong to $V_\delta$) 
we shall encounter the same problems to define in $V$ the class of $\Gamma$-valid
statements (according to our definition) 
as we do have troubles to define in $V$ the set of $V$-truths.}.
Prima facie $\Gamma$-logics 
appear to be even more intractable than $\beta$-logic (the logic given by the 
class of well founded models of $\ZFC$). However this is a misleading point of view, and,
as we shall see below, it is more correct to view these logics as means to radically change our point of view on forcing:
\begin{quote}
$\Gamma$-logics 
transform forcing in a tool to prove theorems over certain natural theories 
$T$ which extend $\ZFC$.
\end{quote}
The following corollary of Cohen's forcing theorem (which we dare to call
Cohen's Absoluteness Lemma) is an illuminating example:
\begin{lemma}[Cohen's Absoluteness]\label{Lem:CohAbs}
Assume $T\supset\ZFC$ and $\phi(x,r)$ is a $\Sigma_0$-formula in the parameter $r$ 
such that $T\vdash r\subset\omega$.
Then the following are equivalent:
\begin{itemize}
\item
$T\vdash[H_{\omega_1}\models\exists x\phi(x,r)]$.
\item
$T\vdash \exists x\phi(x,r)$\emph{ is $\Omega$-consistent\footnote{I.e.
$T\vdash$\emph{ There is a partial order $\mathbb{B}$ such that 
$\Vdash_{\mathbb{B}}\exists x\phi(x,r)$}.}}.
\end{itemize}
\end{lemma}
Observe that for any model $V$ of $\ZFC$,
$H_{\omega_1}^V\prec_{\Sigma_1} V$ and that for 
any theory $T\supseteq\ZFC$ there is a recursive translation of
 $\Sigma^1_2$-properties (provably $\Sigma^1_2$ over $T$)
 into $\Sigma_1$-properties over $H_{\omega_1}$ (provably $\Sigma_1$ over the same theory $T$)
\cite[Lemma 25.25]{JEC03}. 
Summing up we get that a $\Sigma^1_2$-statement is provable in some theory $T\supseteq\ZFC$ 
iff the  corresponding $\Sigma_1$-statement over $H_{\omega_1}$ is 
provably $\Omega$-consistent over the same theory $T$. 
This shows that already in $\ZFC$ forcing is an extremely
powerful tool to prove theorems. 
Moreover compare Lemma~\ref{Lem:CohAbs} 
with Shoenfield's absoluteness theorem stating that
the truth value of a $\Sigma^1_2$-property 
is the same in all transitive models $M$ of $\ZFC$ to which
$\omega_1$ belongs~\cite[Theorem 25.20]{JEC03}. 
These two results are very similar in nature but the first one is more constructive. 
For example a proof that a $\Sigma^1_2$-property holds in $L$
does not yield automatically
that this property is provable in $\ZFC$ but just that it holds in all uncountable transitive models of
$\ZFC$ to which $\omega_1$ belongs; 
yet this property could fail in some non-transitive model of $\ZFC$ or in some 
transitive model of $\ZFC$ whose ordinals have order type at most $\omega_1$.

We briefly sketch why Lemma~\ref{Lem:CohAbs}
holds since this will outline many of the ideas we are heading for:
\begin{proof}
We shall actually prove the following slightly
stronger formulation\footnote{In the statement below we do not require that
the existence of a 
partial order witnessing the $\Omega$-consistency of
$\exists x\phi(x,r)$ in $V$ is provable in $T$.} of the non-trivial direction in the 
equivalence:
\begin{quote}
Assume $V$ is a model of $T$. Then
$H_{\omega_1}\models\exists x\phi(x,r)$ if and only if 
$V\models \exists x\phi(x,r)$\emph{ is $\Omega$-consistent}. 
\end{quote}
To simplify 
the exposition we prove it with the further assumption that 
$V$ is a \emph{transitive} model.
With the obvious care in details essentially the same argument works for any first order
model of $T$.
So assume $\phi(x,\vec{y})$ is a $\Sigma_0$-formula and
$\exists x\phi(x,\vec{r})$ is 
$\Omega$-consistent in $V$ with parameters $\vec{r}\in \mathbb{R}^V$.
Let $\mathbb{P}\in V$ be a partial order that witnesses it. 
Pick a model $M\in V$ such that 
$M\prec (H_{|\mathbb{P}|^+})^V$, $M$ is countable in $V$,  and
$\mathbb{P},\vec{r}\in M$. 
Let $\pi_M:M\to N$ be its transitive collapse
and $\mathbb{Q}=\pi_M(\mathbb{P})$. Notice also that 
$\pi(\vec{r})=\vec{r}$.
Since $\pi_M$ is an isomorphism of $M$ with $N$,
\[
N\models(\Vdash_{\mathbb{Q}}\exists x\phi(x,\vec{r})).
\] 
Now let $G\in V$ be $N$-generic for $\mathbb{Q}$ ($G$ exists since $N$ is countable),
then, by Cohen's fundamental theorem of forcing applied in $V$ to $N$, 
we have that $N[G]\models\exists x\phi(x,\vec{r})$.
So we can pick $a\in N[G]$ such that $N[G]\models\phi(a,\vec{r})$. 
Since $N,G\in (H_{\aleph_1})^V$, we have that
$V$ models that $N[G]\in H_{\omega_1}^V$ and thus $V$ models that 
$a$ as well belongs to $H_{\omega_1}^V$.
Since $\phi(x,\vec{y})$ is a $\Sigma_0$-formula, 
$V$ models that $\phi(a,\vec{r})$ is absolute between the
transitive sets $N[G]\subset H_{\omega_1}$
to which $a,\vec{r}$ belong. In particular
$a$ witnesses in $V$ that $H_{\omega_1}^V\models\exists x\phi(x,\vec{r})$. 
\end{proof}
If we analyze the proof of this Lemma, we immediately realize that a key observation
is the fact
that for any poset
$\mathbb{P}$ there is some countable $M\prec H_{|\mathbb{P}|^+}$ such that
$\mathbb{P}\in M$ and there is an $M$-generic filter for $\mathbb{P}$. 
The latter statement is an easy outcome of Baire's category theorem and is provable in $\ZFC$.
For a given regular cardinal $\lambda$ and a partial order $\mathbb{P}$,
let  $S^\lambda_{\mathbb{P}}$ be the set consisting of
$M\prec H_{\max(|\mathbb{P}|^+,\lambda)}$ such that there is an $M$-generic filter for
$\mathbb{P}$ and $M\cap\lambda\in\lambda>|M|$.
Then an easy outcome of Baire's category theorem is that
$S^{\aleph_1}_{\mathbb{P}}$ is a club subset of $P_{\omega_1}(H_{|\mathbb{P}|^+})$ 
for every partial order $\mathbb{P}$. 
If we analyze the above proof what we actually needed was just the
stationarity of $S^{\aleph_1}_{\mathbb{P}}$ to infer the existence of the desired countable model 
$M\prec H_{|\mathbb{P}|^+}$ such that $r\in M$ and there is an $M$-generic filter for $\mathbb{P}$.
For any regualr cardinal $\lambda$, let $\Gamma_\lambda$ be the class of posets 
such that $S^\lambda_{\mathbb{P}}$ is stationary.
In particular we can generalize Cohen's absoluteness Lemma as follows:
\begin{lemma}[Generalized Cohen Absoluteness]\label{Lem:GenCohAbs}
Assume $V$ is a model of $\ZFC$
and $\lambda$ is regular and uncountable in $V$.
Then $H_{\lambda}^V\prec_{\Sigma_1} V^P$ if 
$P\in\Gamma_{\lambda}$.
\end{lemma}
Let $\FA_\nu(\mathbb{P})$ assert that: 
\emph{$P$ is a partial order such that for every collection of 
$\nu$-many dense subsets of $P$ there is a filter $G\subset P$ 
meeting all the dense 
sets in this collection.} 
Let $\BFA_\nu(\mathbb{P})$ assert that
$H_{\nu^+}^V\prec_{\Sigma_1} V^P$.

Given a class of posets $\Gamma$, let
$\FA_\nu(\Gamma)$ ($\BFA_\nu(\Gamma)$) hold 
if $\FA_\nu(P)$ ($\BFA_\nu(P)$) holds for all $P\in \Gamma$.
Then Baire's category theorem just says that 
$\FA_{\aleph_0}(\Omega)$ holds where
$\Omega$ is the class of all posets.
It is not hard to check that if $S^\lambda_P$ is stationary, then $\FA_\gamma(P)$
holds for all $\gamma<\lambda$.
Woodin~\cite[Theorem2.53]{woodin} 
proved that if $\lambda=\nu^+$ is a successor cardinal
\emph{$P\in\Gamma_\lambda$ if and only if
$\FA_\nu(P)$ holds} 
(see for more details subsection~\ref{sec:forcax} and Lemma~\ref{lem:WOOFA}).
In particular for all cardinals $\nu$ 
we get that $\Gamma_{\nu^+}$ is the class of partial orders
$P$ such that $\FA_\nu(P)$ holds or (equivalently) such that
$S^{\nu^+}_P$ is stationary. 
With this terminology Cohen's
absoluteness Lemma states that
$\FA_\nu(P)$ implies $\BFA_\nu(P)$ for all infinite 
cardinals $\nu$.

Observe that many interesting problems of set theory can be formulated as $\Pi_2$-properties
of $H_{\nu^+}$ for some cardinal
$\nu$ (an example is Suslin's hypothesis, which
can be formulated as a $\Pi_2$-property of $H_{\aleph_2}$). 
Lemma~\ref{Lem:GenCohAbs}
gives a very powerful general framework to prove in any given model $V$ of $\ZFC$
whether a $\Pi_2$-property $\forall x\exists y\phi(x,y,z)$ (where $\phi$ is $\Sigma_0$) holds for
some $H_{\nu^+}^V$ with $p\in H_{\nu^+}^V$ replacing $z$:
It suffices to prove that for any $a\in H_{\nu^+}^V$,
$V$ models that $\exists y\phi(a,y,p)$ is $\Gamma_{\nu^+}$-consistent.
This shows that if we are in a model $V$ of $\ZFC$
where $\Gamma_{\nu^+}^V$ contains interesting and manageable 
families of partial orders
$\Gamma_{\nu^+}^V$-logic is a powerful tool to study the 
$\Pi_2$-theory of $H_{\nu^+}^V$. 
In particular this is always the case for $\nu=\aleph_0$ in any model of
$\ZFC$, since 
$\Gamma_{\aleph_1}$ is the class of all posets. 
Moreover this is certainly
one of the reasons of the success the forcing 
axiom Martin's Maximum $\MM$ and its bounded version
$\BMM$ have had in settling many relevant problems of set theory
which can be formulated as $\Pi_2$-properties of the structure $H_{\aleph_2}$
 and that boosted the study of bounded versions of forcing 
 axioms\footnote{Bagaria~\cite{BAG00} and Stavi, V{\"a}{\"a}n{\"a}nen~\cite{STAVAA02} 
 are the first who realize that bounded forcing axioms 
 are powerful tools to describe the $\Pi_2$-theory of $H_{\aleph_2}$ exactly for the
 reasons we are pointing out.}.

For any set theorist willing to accept large cardinal axioms,
Woodin has been able to show that
$\Omega$-logic gives a natural non-constructive semantics for
the full first order theory 
of $L(\mathbb{R})$ and not just for the $\Sigma_1$-fragment of 
$H_{\aleph_1}\subset L(\mathbb{R})$ which is given by Cohen's absoluteness 
Lemma.
Woodin~\cite[Theorem 2.5.10]{LAR04} has proved that assuming large  
cardinals
\emph{$\Omega$-truth is $\Omega$-invariant} i.e.: 
\begin{quote}
Let $V$ be any model of $\ZFC+$\emph{there are
class many Woodin cardinals}.
Then for any statement $\phi$ with parameters in $\mathbb{R}^V$, 
\[
V\models(\phi\text{ is }\Omega\text{-consistent})
\]
if and only if there is $\mathbb{B}\in V$ such that
\[
V^{\mathbb{B}}\models(\phi\text{ is }\Omega\text{-consistent}). 
\]
\end{quote}
Thus $\Omega$-logic, the logic of forcing, has a notion of truth which forcing itself cannot change. 
Woodin~\cite[Theorem 3.1.7]{LAR04} also proved that 
the theory $\ZFC+$\emph{large cardinals} 
decides in $\Omega$-logic the theory of $L(\mathbb{R})$, i.e.:  
\begin{quote}
For any model $V$ of $\ZFC+$\emph{there are
class many Woodin cardinals which are a limit of Woodin cardinals} 
and any first order formula $\phi$, 
$L(P_{\omega_1}\Ord)^V\models\phi$
if and only if 
\[
V\models[L(P_{\omega_1}\Ord)\models\phi] \text{ is $\Omega$-consistent}.
\]
\end{quote}
He pushed further these result and showed that if $T$ extends
$\ZFC+$\emph{ There are class many measurable Woodin cardinals}, then $T$ 
decides in $\Omega$-logic 
any mathematical problem expressible as a (provably in $T$) 
$\Delta^2_1$-statement. 
These are optimal and sharp results: it is well known that the Continuum hypothesis $\mathsf{CH}$ 
(which is provably not a $\Delta^2_1$-statement) 
and the first order theory of $L(P(\omega_1))$ cannot be decided by 
$\ZFC+$\emph{ large cardinal axioms} in $\Omega$-logic. 
Martin and Steel's result that projective determinacy holds in $\ZFC^*$ complements  
the fully satisfactory description $\Omega$-logic and large cardinals
give of the first order theory of the structure 
$L(\mathbb{R})$ in models of $\ZFC^*$. 
Moreover we can make these results meaningful also for a non-platonist, 
for example we can reformulate the statement that $\ZFC^*$ decides in 
$\Omega$-logic the theory of $L(\mathbb{R})$ 
as follows:
\begin{quote}
Assume $T$ extends $\ZFC+$\emph{there are
class many Woodin cardinals which are a limit of Woodin cardinals}. 
Let $\phi(r)$ be a formula in the parameter $r$ such that
$T\vdash r\subseteq\omega$. Then the following are equivalent:
\begin{itemize}
\item
$T\vdash[L(P_{\omega_1}\Ord)\models\phi(r)]$.
\item
$T\vdash\phi(r)^{L(P_{\omega_1}\Ord)}$\emph{ is $\Omega$-consistent.}
\end{itemize}
\end{quote}

The next natural stage is to determine
to what extent Woodin's results on $\Omega$-logic and the theory of $H_{\aleph_1}$
and $L(\mathbb{R})$ 
can be reproduced for $H_{\aleph_2}$ and $L(P(\omega_1))$.
There is also for these theories a fundamental
result of Woodin: he introduced an axiom $(*)$ 
which is a strengthened version of $\BMM$ with the property that 
the theory of $H_{\aleph_2}$ with parameters 
is invariant with respect to \emph{all} forcings which preserve this 
axiom\footnote{We refer the reader to~\cite{LARHST} for a thorough development of
the properties of models of the $(*)$-axiom.}.
The $(*)$-axiom is usually formulated~\cite[Definition 7.9]{LARHST} 
as the assertion that 
$L(\mathbb{R})$ is a model of the axiom of determinacy and
$L(P(\omega_1))$ is a generic extension of $L(\mathbb{R})$ by
an homogeneous forcing $\mathbb{P}_{\max}\in L(\mathbb{R})$. 
 
There are two distinctive features of $(*)$: 
\begin{enumerate}
\item
It asserts a 
smallness principle for
$L(P(\omega_1))$: on the one hand the homogeneity of 
$\mathbb{P}_{\max}$ entails that the first order theory of $L(P(\omega_1))$
is essentially determined by the theory of the underlying $L(\mathbb{R})$.
On the other hand $(*)$ implies that $L(P(\omega_1))=L(\mathbb{R})[A]$ for any
$A\in P(\omega_1)\setminus L(\mathbb{R})$.

\item
$(*)$ entails that $H_{\omega_2}^V\prec H_{\omega_2}^{V^P}$ for any notion of 
forcing $P\in V$ which preserves $(*)$ even if $\FA_{\aleph_2}(P)$ may be false
for such a $P$. 
\end{enumerate}

In this paper we propose a different approach to the analysis of the theory of
$H_{\aleph_2}$ then the one given by $(*)$. 
We do not seek for an axiom system $T\supseteq \ZFC$ which makes
the theory of $H_{\aleph_2}$ invariant with respect to \emph{all} forcing
notions which preserve a suitable fragment of $T$.
Our aim is to show that the strongest forcing axioms in combination with large cardinals
give an axiom system $T$ which extends $\ZFC$ and
makes the theory of $H_{\aleph_2}$ invariant with respect to all forcing 
notions $P$ which preserve a suitable fragment of $T$ 
\emph{and for which we can predicate $\FA_{\aleph_1}(P)$} (i.e. forcings $P$ which are in the 
class  $\Gamma_{\aleph_2}$).

This leads us to analyze the properties of the class $\Gamma_{\aleph_2}$ in models of
$\ZFC^*$.
This is a delicate matter, first of all 
Shelah proved that $\FA_{\aleph_1}(P)$ fails for any $P$ which does not 
preserve stationary subsets of $\omega_1$. Nonetheless it cannot be decided in $\ZFC$ 
whether this is a necessary condition for a poset $P$ in order to have the failure of
$\FA_{\aleph_1}(P)$.
For example let $P$ be Namba forcing: it is provable in $\ZFC$ that $P$ 
preserve stationary subsets 
of $\omega_1$, however 
in $L$ $\FA_{\aleph_1}(P)$ fails while in a
model of Martin's maximum $\MM$ $\FA_{\aleph_1}(P)$ holds.
This shows that we cannot hope to prove general theorems about $H_{\aleph_2}$
in $\ZFC^*$ alone using forcing, but just theorems about the properties of
$H_{\aleph_2}$ for 
particular theories $T$ which extend $\ZFC^*$ and
for which we have a nice description 
of the class $\Gamma_{\aleph_2}$.

In this respect it is well known that the study of the properties of
$H_{\aleph_2}$ in models of Martin's maximum $\MM$, 
of the proper forcing axiom $\PFA$, or of their bounded
versions $\BMM$ and $\BPFA$ has been particularly 
successful. Moreover it is well known that the strongest such theories 
($\MM$ and $\PFA$) are able to settle many relevant questions about the whole universe $V$
and to show that many properties of the universe reflect to 
$H_{\aleph_2}$\footnote{The literature is vast, we mention just a sample of the most 
recent results with no hope of being 
exhaustive:~\cite{MOO06,TOD02,woodin} present different examples of
well-ordering of the reals definable in $H_{\aleph_2}$ (with parameters in $H_{\aleph_2}$)
in models of $\BMM$ ($\BPFA$),~\cite{COX11a,VIA11b,VIAWEI11} present
several different reflection properties between the universe and $H_{\aleph_2}$ in models of
$\MM^{++}$ ($\PFA$,$\MM$),~\cite{FAR11,moore.basis} present applications of $\PFA$
to the solution of problems coming from operator algebra and general topology
and which can be formulated as (second order) properties of the structure $H_{\aleph_2}$.}. 
The reason is at least two-fold:
\begin{itemize}
\item First of all there is a manageable description 
of the class $\Gamma_{\aleph_2}$ in models of 
$\MM$ ($\PFA$,$\MA$): this is the class of stationary set preserving posets for
$\MM$ (respectively contains the class of proper forcings for $\PFA$, and the class of 
CCC partial orders for $\MA$).


\item $\MM$
realizes the slogan that 
$\FA_{\aleph_1}(P)$ holds
for any partial order $P$ for which we cannot prove that $\FA_{\aleph_1}(P)$ 
fails, thus $\MM$ substantiates a natural 
maximality principle for the class $\Gamma_{\aleph_2}$.
\end{itemize}

We believe that the arguments we presented so far already show that for any model $V$ of
$\ZFC$ and any successor cardinal $\lambda\in V$ it is of central interest to analyze what is the class
$\Gamma_\lambda$ in $V$, since this gives a powerful tool to 
ivestigate the $\Pi_2$-theory
of $H_\lambda^V$. Moreover in this respect $\ZFC+\MM$ is particularly appealing
since it asserts the maximality of the class $\Gamma_{\aleph_2}$. 
The main result of this paper is to show that a natural strengthening of $\MM$ (denoted by 
$\MM^{++}$) which holds in the standard models of $\MM$, 
in combination with Woodin cardinals, makes $\Gamma_{\aleph_2}$-logic 
the correct semantics to describe completely the $\Pi_2$-theory of $H_{\aleph_2}$ in models
of $\MM^{++}$. In particular we shall prove the following theorem:
\begin{theorem}\label{thm:absMMLORD-0}
Assume
$\MM^{++}$ holds and 
there are class many Woodin cardinals.
Then 
\[
H_{\aleph_2}^V\prec_{\Sigma_2} H_{\aleph_2}^{V^P}
\]
for all stationary set preserving posets $P$ which preserve $\BMM$ 
\end{theorem}
Notice that we can reformulate the theorem in the same fashion of 
Woodin's and Cohen's results as follows:
\begin{theorem}\label{thm:absMMLORD-1}
Assume $T$ extends 
$\ZFC+\MM^{++}+$\emph{ There are class many Woodin cardinals}.
Then for every $\Pi_2$-formula $\phi(x)$ in the free variable $x$
and every parameter $p$ 
such that $T\vdash p\in H_{\omega_2}$ the following are equivalent: 
\begin{itemize}
\item
$T\vdash[H_{\aleph_2}\models\phi(p)]$ 
\item
$T\vdash$\emph{ There is a stationary set preserving partial order $P$ such that
$\Vdash_P \phi^{H_{\aleph_2}}(p)$ 
and $\Vdash_P\BMM$}.
\end{itemize}
\end{theorem}

The rest of this paper is organized as follows: Section~\ref{sec:prel} 
presents background material on forcing (Subsection~\ref{sec:comemb}), 
forcing axioms (Subsection~\ref{sec:forcax}), the stationary tower
forcing (Subsection~\ref{sec:stattow}),
the relation between the stationary tower forcing and forcing axioms
(Subsection~\ref{subsec:prel3}), and a new 
characterization of the forcing axiom $\MM^{++}$ in terms of
complete embeddings of stationary set preserving posets into stationary 
tower forcings (Subsection~\ref{subsec:prel4}).
Section~\ref{sec:mainth} presents the proof of the main result, while 
Section~\ref{sec:WooHC} gives a proof of the invariance of the theory of 
$H_{\aleph_1}$ with respect to set forcing in the presence of class many Woodin cardinals.
 We end the paper with some comments and open 
questions (Section~\ref{sec:comop}).

While the paper is meant to be as much self-contained as possible, 
we presume that familiarity with forcing axioms (in particular with Martin's maximum)
and with the stationary tower forcing 
are of valuable help for the reader.
A good reference for background material on  Martin's maximum
is~\cite[Chapter 37]{JEC03}. For the stationary tower forcing a reference text is~\cite{LAR04}.

\subsection{Notation and prerequisites}\label{subsec:not}
We adopt standard notation which is customary in the subject, our reference text is~\cite{JEC03}.

For models $(M,E)$ of $\ZFC$, we say that $(M,E)\prec_{\Sigma_n} (M',E')$ if
$M\subset M'$, $E=E'\cap M^2$ and for any $\Sigma_n$-formula $\phi(p)$ 
with $p\in M$, $(M,E)\models\phi(p)$ 
if and only if $(M',E')\models\phi(p)$. 
We usually write $M\prec_{\Sigma_n} M$ instead of  $(M,E)\prec_{\Sigma_n} (M',E')$ when
$E,E'$ is clear from the context. We let $(M,E)\prec (M',E')$ if  
$(M,E)\prec_{\Sigma_n} (M',E')$ for all $n$.

We let $\Ord$ denote the class of ordinals. For any cardinal $\kappa$ $P_\kappa X$ 
denote the subsets of $X$ of size less than $\kappa$.
Given $f:X\to Y$ and $A\subset X, B\subset Y$, $f[A]$ is the pointwise image of $A$ under $f$ and
$f^{-1}[B]$ is the preimage of $B$ under $f$.
A set $S$ is stationary if for all $f:P_{\omega}(\cup S)\to \cup S$ there is $X\in S$ such that
$f[X]\subseteq X$ (such an $X$ is called a closure point for $f$).
A set $C$ is a club subset of $S$ if it meets all stationary subsets of $S$ or, equivalently, if
it contains all the closure points in $S$ of some $f:P_{\omega}(\cup S)\to \cup S$.
Notice that $P_\kappa X$ is always stationary if $\kappa$ is a cardinal and $X$, $\kappa$ are 
both uncountable.

If $V$ is a transitive model of $\ZFC$ and $(P,\leq_P)\in V$ is a partial order with a top element
$1_P$, 
$V^P$ denotes the class of $P$-names, and $\dot{a}$ or $\tau$ denote an arbitrary element of 
$V^P$, if
$\check{a}\in V^P$ is the canonical name for a set $a$ in $V$ we drop 
the superscript and confuse
$\check{a}$ with $a$.
We also feel free to confuse the approach to forcing via boolean valued models
as done by Scott and others or via the forcing relation. Thus we shall write for example
$V^P\models\phi$ as an abbreviation for 
\[
V\models[1_P\Vdash \phi].
\]

If $M\in V$ is such that $(M,\in)$ is 
a model of a sufficient fragment of $\ZFC$ and 
$(P,\leq_P)$ in $M$ is a partial order, an $M$-generic filter for 
$P$ is a filter $G\subset P$ 
such that $G\cap A\cap M$ is non-empty for all maximal antichains $A\in M$ (notice that
if $M$ is non-transitive, $A\not\subseteq M$ is well possible). 
If $N$ is a \emph{transitive} model of a large enough fragment of
$\ZFC$, $P\in N$ and $G$ is an $N$-generic filter for $P$, let 
$\sigma_G:N\cap V^P\to N[G]$ denote 
the evaluation map induced by $G$ of the $P$-names in $N$.

We say that $(M,E)\prec_{\Sigma_n} (\dot{N},\dot{E})$ for some $\dot{N}\in V^P$
if 
\[
V^P\models \dot{E}\cap M^2=E
\] 
and for any $\Sigma_n$-formula $\phi(p)$ 
with $p\in M$, $(M,E)\models\phi(p)$ 
if and only if 
\[
V^P\models [(\dot{N},\dot{E})\models\phi(p)].
\] 
We will write $M\prec_{\Sigma_n} \dot{N}$ if
$(M,E)\prec_{\Sigma_n} (\dot{N},\dot{E})$ and $E,\dot{E}$ are clear from the context.

We shall also frequently refer to Woodin cardinals,
however for our purposes we won't need to recall the definition of a 
Woodin cardinal but just its effects on the  properties of the
stationary tower forcing. This is done in 
subsection~\ref{sec:stattow}.

\section{Preliminaries}\label{sec:prel} 

We shall briefly outline some general results on the theory of forcing
which we shall need for our exposition.
The reader may skip Subsections~\ref{sec:comemb},~\ref{sec:forcax},~\ref{sec:stattow}
and eventually refer back to them.

\subsection{Preliminaries \textsf{I}: complete embeddings and projections}\label{sec:comemb}

For a poset $Q$ and $q\in Q$, let $Q\restriction q$ denote the poset $Q$ restricted to conditions $r\in Q$ which are below $q$
 and $\mathbb{B}(Q)$ denote its boolean completion, i.e. the complete boolean algebra of regular open subsets of $Q$, so that
 $Q$ is naturally identified with a dense subset of $\mathbb{B}(Q)$.
 We say that: 
\begin{itemize}
\item
$P$ completely embeds into $Q$ if   
there is an homomorphism $i:P\to \mathbb{B}(Q)$ which preseves the order relation and maps maximal antichains of 
$P$ into maximal antichains of $\mathbb{B}(Q)$. With abuse of notation we shall call a complete embedding
of $P$ into $Q$  any such homomorphism $i:P\to \mathbb{B}(Q)$.
\item
$i:P\to \mathbb{B}(Q)$ is locally complete  if
for some $q\in Q$, $i:P\to \mathbb{B}(Q\restriction q)$ is a complete embedding
(with a slight abuse of the current terminology, we shall also call any locally complete embedding a 
\emph{regular} embedding).
\item
$P$ projects to $Q$ if there is an order preserving map $\pi:P\to Q$ whose image is dense in $Q$.
\end{itemize}

\begin{lemma}\label{lem:cige}
The following are equivalent:
\begin{enumerate}
\item \label{lem:cige-1}
$P$ completely embeds into $Q$,
\item\label{lem:cige-2}
for any $V$-generic filter $G$ for $Q$ 
there is in $V[G]$ a $V$-generic filter $H$ for $P$,
\item\label{lem:cige-3}
For some $p\in P$ there is a homomorphism 
$i:\mathbb{B}(P\restriction p)\to\mathbb{B}(Q)$
of complete boolean algebras.
\end{enumerate}
\end{lemma}

\begin{proof}
We proceed as follows:
\begin{description}
\item
\textbf{\ref{lem:cige-1} implies~\ref{lem:cige-2}}

Observe that if 
$i:P\to \mathbb{B}(Q)$ is a complete embedding and $G$ is
a $V$-generic filter for $\mathbb{B}(Q)$, then
$H=i^{-1}[G]$ is a $V$-generic filter for $P$.

\item
\textbf{\ref{lem:cige-2} implies~\ref{lem:cige-1}}

Let
$\dot{H}\in V^{\mathbb{B}(Q)}$ be a name such that
\[
\Vdash_{\mathbb{B}(Q)}\dot{H}\text{ is a $V$-generic filter for }P.
\] 
The map $p\mapsto \|\check{p}\in\dot{H}\|_{\mathbb{B}(Q)}$ 
is the desired complete embedding of $P$ into $Q$.

\item
\textbf{\ref{lem:cige-1} implies~\ref{lem:cige-3}}

Let
$i:P\to \mathbb{B}(Q)$ be a complete embedding
and $\dot{H}\in V^{\mathbb{B}(Q)}$ be a name for the
$V$-generic filter for $\mathbb{B}(Q)$.
Then there is some $p\in P$ such that 
\[
\|i(q)\in\dot{H}\|_{\mathbb{B}(Q)}>0_{\mathbb{B}(Q)}
\]
for all $q\leq p$.
Then for such a $p$ the map $i$ can naturally be extended to a complete homomorphism 
$i:\mathbb{B}(P\restriction p)\to\mathbb{B}(Q)$.

\item
\textbf{\ref{lem:cige-3} implies~\ref{lem:cige-1}}

Immediate.
\end{description}
\end{proof}

\begin{remark}
Observe that if $i:P\to \mathbb{B}(Q)$ is a complete embedding then for all
$q\in Q$ such that $i(p)\wedge q> 0_{\mathbb{B}}$, 
the map $i_q:P\to \mathbb{B}(Q\restriction q)$ which maps $p$ to $q\wedge i(p)$ is also a complete 
embedding.
Moreover if $q\Vdash_Q \check{p}\in\dot{H}$ where $\dot{H}=i^{-1}[\dot{G}]\in V^{Q}$ and 
$\dot{G}$ is the canonical $\mathbb{B}(Q)$-name for 
a $V$-generic filter for $\mathbb{B}(Q)$, we have that 
$i_q(r)=0_{\mathbb{B}(Q)}$ for all $r\in P$ incompatible with $p$.

Thus in general a complete embedding $i:P\to \mathbb{B}(Q)$ may map a large portion of
$P$ to $0_{\mathbb{B}(Q)}$. 
\end{remark}

\begin{lemma} \label{lem:prge}
The following are equivalent:
\begin{enumerate}
\item \label{lem:prge-1}
There  is a projection $\pi:P\to \mathbb{B}(Q)\setminus\{0_{\mathbb{B}(Q)}\}$.
\item  \label{lem:prge-2}
There is $\dot{H}$ in $V^{P}$ such that
$\Vdash_P\dot{H}$\emph{ is a $V$ generic filter for $Q$}.
\end{enumerate}
\end{lemma}

\begin{proof}
\begin{description}
\item
\textbf{\ref{lem:prge-1} implies~\ref{lem:prge-2}}

Let $\dot{H}\in V^P$ be a $P$-name such that $\Vdash_P\dot{H}=\pi[\dot{G}]$.
Then since the image of $\pi$ is a dense subset of $\mathbb{B}(Q)$ it is easy to check that
$\Vdash_P\dot{H}$\emph{ generates a $V$-generic filter for $\mathbb{B}(Q)$}.
\item
\textbf{\ref{lem:prge-2} implies~\ref{lem:prge-1}}
Assume~\ref{lem:prge-2} holds for the $P$-name $\dot{H}$ and let
$\pi:P\to\mathbb{B}(Q)$ be defined by $\pi(p)=\bigwedge\{q\in Q: p\Vdash_P q\in\dot{H}\}$.
We claim that $\pi$ is a projection. First of all we claim that
$\pi(p)>0_{\mathbb{B}(Q)}$ for all $p\in P$.
This uses the following observation:
\begin{fact}
Assume $G$ is $V$-generic for $P$ and 
$H=\sigma_G(\dot{H})\in V[G]$ is $V$-generic for $Q$.
If $A\in V$ is such that $A\subset H$, then $\bigwedge A>0_{\mathbb{B}(Q)}$.
\end{fact}
\begin{proof}
Assume not, then there is some $r\in G$ such that $r\Vdash_P \bigwedge A=0_{\mathbb{B}(Q)}$.
Now let $A=\{a_i:i\in I\}$. Since $A\subset G$ all the $a_i$ are compatible.
Let $B\subset A$ $B\in V$ be a non-empty subset of $A$ of 
least size such that $\bigwedge B=0_{\mathbb{B}(Q)}$ but for no
$E\subset B$ such that $|E|<|B|$, $\bigwedge E=0_{\mathbb{B}(Q)}$.
Then we can arrange that $B=\{a_\alpha:\alpha<\gamma\}$ for some cardinal $\gamma$ and that
$b_\beta=\bigwedge\{a_\alpha:\alpha<\beta\}>0_{\mathbb{B}(Q)}$ for all $\beta<\gamma$.
By refining
the sequence $\{b_\beta:\beta<\gamma\}$ (if necessary) we can further suppose that $b_\alpha<b_\beta$ for
all $\alpha<\beta$.

Now let $c_\beta=b_0\wedge \neg b_\beta$.
We claim that 
\[
C=\{c:\exists \beta<\gamma\text{ such that } c\leq c_\beta\}
\] 
(which belongs to $V$) 
is a dense subset of
$\mathbb{B}(Q)\restriction b_0$. 
To see this, let $\dot{H}_0$ be a $P$-name which is forced by $P$
to be
the $V$-generic filter for $\mathbb{B}(Q)$ generated by $\dot{H}$
which is a $P$-name for a $V$-generic filter for $Q$.
 Observe that since
$r\Vdash_P \dot{H}_0$\emph{ is a $V$-generic filter for 
$\mathbb{B}(Q)$ containing $\dot{H}$}, 
we get that $r\Vdash_P b_\beta\in\dot{H}_0$ for all $\beta<\gamma$. 
Now given some $c\leq b_0$,  
we have that $\bigwedge\{c\wedge b_\beta:\beta<\gamma\}=0_{\mathbb{B}(Q)}$.
There are two cases:
\begin{itemize}
\item 
There is some $\beta<\gamma$ such that 
$c\wedge b_\delta=c\wedge b_\beta$ for all
$\delta\in[\beta,\gamma)$. 

In this case we have that $c\wedge b_\beta=0_{\mathbb{B}(Q)}$
and $c$ is already an element of $C$.

\item For all $\beta<\gamma$ there is $\alpha_\beta<\gamma$ such that 
$c\wedge b_{\alpha_\beta}< c\wedge b_\beta$.

In this case we get that  $c\wedge b_{\alpha_0}< c$. 
Thus $d=c\wedge\neg  b_{\alpha_0}>0_{\mathbb{B}(Q)}$ and
$d\leq c$ is an element of $C$. 

\end{itemize}
Thus $C$ is dense.
Since $r\Vdash_P C\cap\dot{H}_0\neq\emptyset$ we can find $r'\leq r$ and $d\in C$ 
such that $r'\Vdash_P d\in C\cap\dot{H}_0$. Now there is some $\beta$ such that 
$d\wedge b_\beta=0_{\mathbb{B}(Q)}$.
Then $r'\Vdash_P d\wedge b_\beta=0_{\mathbb{B}(Q)}\in\dot{H}_0$, 
which is the desired contradiction which proves 
the fact. 

\end{proof}
Now for all $p\in P$, $A_p=\{q\in Q:p\Vdash q\in\dot{H}\}$ is in $V$ 
and is forced by $p$ to be a subset of
$\dot{H}$. In particular we get that $\bigwedge A_p=\pi(p)>0_{\mathbb{B}(Q)}$.
iI is now easy to check that $\pi[P]$ is a dense subset of 
$\mathbb{B}(Q)\setminus \{0_{\mathbb{B}(Q)}\}$.  
\end{description}
\end{proof}

Given  
a complete embedding $i:\mathbb{Q}\to\mathbb{B}$ of complete boolean algebras,
let $\pi:\mathbb{B}\to\mathbb{Q}$ map
$a$ to $\inf\{q\in\mathbb{Q}: i(q)\geq a\}$, then $\pi$ is a projection and
$\pi\circ i(b)=b$ for all $b\in \mathbb{B}$ while $i\circ \pi(q)\geq q$ for all $q\in \mathbb{Q}$. 

The quotient forcing $\mathbb{B}/i[\mathbb{Q}]$ is the object belonging to
$V^{\mathbb{Q}}$ such that
\begin{itemize}
\item
$\Vdash_{\mathbb{Q}}\mathbb{B}/i[\mathbb{Q}]$\emph{ is a partial order 
with the order relation $\leq_i$.}
\item
$\mathbb{B}/i[\mathbb{Q}]\in V^{\mathbb{Q}}$ is the set of
$\mathbb{Q}$-names $\dot{r}$ of least rank among those that satisfy the following requirements:
\begin{itemize}
\item
$\llbracket\dot{r}\in(\check{\mathbb{B}}\setminus\{0_{\mathbb{B}}\})\rrbracket_{\mathbb{Q}}
=1_{\mathbb{Q}}$.
\item
For all $\dot{r}\in\mathbb{B}/i[\mathbb{Q}]$ if there are $r\in\mathbb{B}$ and $q\in\mathbb{Q}$
such that $q\Vdash_{\mathbb{Q}}\dot{r}=r$, then $\pi(r)\geq q$.
\end{itemize}
\item
For $\dot{r},\dot{s}\in \mathbb{B}/i[\mathbb{Q}]$
$q\Vdash_{\mathbb{Q}}\dot{r}\leq_i \dot{s}$ if and only if the following holds:

\begin{quote}
For all $q'\leq q$, if there are $r,s\in\mathbb{B}$ such that
$q'\Vdash_{\mathbb{Q}}\dot{r}=r\wedge\dot{s}=s$, then 
$r\wedge i(q')\leq_{\mathbb{B}}s\wedge i(q')$.
\end{quote}

\end{itemize}

\begin{lemma}
If $i:\mathbb{Q}\to\mathbb{B}$ is a complete embedding of 
complete boolean algebras, then $\mathbb{Q}*(\mathbb{B}/i[\mathbb{Q}])$ is forcing equivalent
to $\mathbb{B}$.
\end{lemma}
\begin{proof}
Let $\pi:\mathbb{B}\to\mathbb{Q}$ be the projection map associated to $i$.
The map 
\[
i^*:(\mathbb{B}\setminus\{0_{\mathbb{B}}\})\to\mathbb{B}(\mathbb{Q}*(\mathbb{B}/i[\mathbb{Q}]))
\]
which maps $r\mapsto(\pi(r),\check{r})$ is a complete embedding such that 
$i^*[\mathbb{B}\setminus\{0_{\mathbb{B}}\}]$ is 
dense in $\mathbb{Q}*(\mathbb{B}/i(\mathbb{Q}))$.

The conclusion follows.
\end{proof}

\begin{remark}
There might be a variety of locally complete embeddings of a poset $P$ into a poset $Q$.
These embeddings greatly affect the properties the generic extensions by $Q$
attributes to elements of the generic extensions by $P$.
For example the following can occur: 
\begin{quote}
There is a $P$-name $\dot{S}$ which is forced by $P$ to be
a stationary subset of $\omega_1$ and there are $i_0:P\to\mathbb{B}(Q)$,
$i_1:P\to\mathbb{B}(Q)$ distinct locally complete embeddings of $P$ into $Q$ such that
if $G$ is $V$-generic for $\mathbb{B}(Q)$ and $H_j=i_j^{-1}[G]$, then 
$\sigma_{H_0}(\dot{S})$ is stationary in $V[G]$, $\sigma_{H_1}(\dot{S})$ is stationary
in $V[H_1]$ but non-stationary in $V[G]$.
\end{quote}
\end{remark}

If $i:P\to \mathbb{B}(Q)$ is a locally complete embedding and $p\in P$, $q\in Q$ are such that
$i$ can be extended to a complete homomorphism of $\mathbb{B}(P\restriction p)$ into
$\mathbb{B}(Q\restriction q)$ we shall also
denote $\mathbb{B}(Q\restriction q)/i[\mathbb{B}(P\restriction p)]$ by $Q/i[P]$, if $i$ is clear from the
context we shall even denote such quotient forcing as $Q/P$.

\subsection{Preliminaries \textsf{II}: forcing axioms}\label{sec:forcax}

\begin{definition}
Given a cardinal $\lambda$ and a partial order $P$,
$\FA_{\lambda}(P)$ holds if:
\begin{quote}
For all $p\in P$, $P\restriction p$ is a partial order such that for every collection of 
$\lambda$-many dense subsets of $P\restriction p$ there is a filter $G\subset P\restriction p$ 
meeting all the dense set in this collection.
\end{quote}

$\FA_{<\lambda}(P)$ holds if $\FA_{\nu}(P)$ holds for all $\nu<\lambda$.

$\BFA_\lambda(P)$ holds if $H_\lambda\prec_{\Sigma_1} V^P$.

If $\Gamma$ is a family of partial orders, $\FA_\lambda(\Gamma)$ ($\FA_{<\lambda}(\Gamma)$,
$\BFA(\Gamma)$)
asserts that
$\FA_\lambda(P)$ ($\FA_{<\lambda}(P)$, $\BFA(P)$)
holds for all $P\in\Gamma$.


\smallskip

For any partial order $P$
\[
S^\lambda_{P}=\{M\prec H_{|P|^+}:M\cap\lambda\in\lambda>|M|\text{ and there is an 
$M$-generic filter for $P$}\} 
\] 

\end{definition}

For any regular uncountable cardinal $\lambda$,
we let $\Gamma_\lambda$ be the family of $P$ such that $S^\lambda_P$ is stationary.

In the introduction we already showed:
\begin{lemma}
Assume $\lambda$ is a regular cardinal. Then
$P\in \Gamma_\lambda$ implies $\BFA_\lambda(P)$.
\end{lemma}

$\MM$ asserts that $\FA_{\aleph_1}(\mathsf{SSP})$ holds, 
where $\mathsf{SSP}$ is the family 
of posets which preserve stationary subsets of $\omega_1$.
$\BMM$ asserts that $\BFA_{\aleph_1}(\mathsf{SSP})$ holds.
It is not hard to see that if $S^\lambda_{P}$ is stationary, then $\FA_{<\lambda}(P)$ holds.
It is not clear whether the converse holds if $\lambda$ is inaccessible. 
However the converse holds if $\lambda$ is a successor cardinal
and Woodin's~\cite[Theorem 2.53]{woodin}
gives a special case of the following Lemma for $\lambda=\omega_2$.
\begin{lemma}\label{lem:WOOFA}

Let $\lambda=\nu^+$ be a
successor cardinal. 
Then the following are equivalent: 
\begin{enumerate}
\item
$\FA_\nu(P)$ holds.
\item
$S_{P}^\lambda$ is stationary.
  
\end{enumerate}

\end{lemma}

\begin{proof}
Only one direction is non trivial. 
We assume that $\FA_\nu(P)$ holds in $V$ and we prove 
that $V$ models that $S_{P}^\lambda$ is stationary.
We leave to the reader to prove the other
implication.

First of all we leave the reader to check that if $\FA_\nu(P)$ holds, then all cardinals less or equal to 
$\nu$ are
preserved by $P$.
Let $P\in H_\theta$ with $\theta$ regular larger than $\lambda$.

Pick $M_0\prec H_\theta$ such that $P\in M_0$ and $M_0\cap\lambda\in\lambda>|M_0|=\nu$.
Now since $|M_0|=\nu$, there is a filter $H$ which meets all the dense sets in $M_0$.
The proof is completed once we prove the following:
\begin{claim}
\[
M_1=\{a\in H_\theta:\exists\tau\in M_0\cap V^P\exists q\in H\text{ such that }
q\Vdash_P a=\tau\}\prec H_\theta,
\]
$H$ is an $M_1$-generic filter for $P$, $|M_1|=\nu$ and $M_1\cap\lambda\in\lambda$.
\end{claim}
\begin{proof}
We prove each item as follows:

\begin{itemize}

\item \textbf{$M_1\prec H_\theta$:}

Given a first order formula  $\phi(x_0,\dots,x_n)$, and $a_1,\dots a_n\in M_1$ such that
$H_\theta$ models $\exists x \phi(x,a_1\dots,a_n)$ we want to find $a_0\in M_1$ such that
$H_\theta$ models $\phi(a_0,a_1\dots,a_n)$. Let $\tau_1,\dots,\tau_n\in M_0\cap V^P$ be 
such that for some $q_i\in H$, $q_i\Vdash \tau_i=a_i$ and 
\[
\Vdash_P\exists x\in H_\theta^V\phi(x,\tau_1,\dots\tau_n)^{H_\theta^V}.
\]
Since $P\in H_\theta$ we can find $\tau\in H_\theta$ such that
\[
\Vdash_P\phi(\tau,\tau_1,\dots\tau_n)^{H_\theta^V}\wedge\tau\in V.
\] 
In particular we get that
\[
H_\theta\models[\Vdash_P \tau\in V\wedge\phi(\tau,\tau_1,\dots,\tau_n)^{H_\theta^V}].
\]
Since $M_0\prec H_\theta$, we can actually find such a $\tau\in M_0\cap V^P$.
Then the set of $q\in P$ which force the value of $\tau$ to be some element of
$H_\theta$ is open dense and belongs to $M_0$.
Thus there is  $q\in H$ which belongs to this
open dense set and refines all the $q_i$, and $a\in H_\theta$  such that
$q\Vdash a=\tau$. Then $a\in M_1$ and 
$H_\theta$ models that $\phi(a,a_1\dots,a_n)$, as was to be shown.

\item \textbf{$H$ is an $M_1$-generic filter for $P$:} 

Pick $D\in M_1$ dense subset of $P$ and
$\dot{D}\in M_0$ such that 
$\Vdash_P\tau$\emph{ is a dense subset of $P$ which belongs to $V$} and
such that for some
$q\in H$, $q\Vdash_P\dot{D}=D$. Then we get that $\Vdash_P \tau\cap\dot{G}\neq\emptyset$, 
thus there is some
$\tau'\in M_0$ such that $\Vdash_P \tau'\in\dot{D}\cap\dot{G}$, since $M_0\prec H_\theta$.
Now we can find $r\leq q$, $r\in H$ and $p\in P$ such that $r\Vdash_P\tau'=p$.
Since 
\[
r\Vdash_P p=\tau'\in\dot{G}\cap \dot{D}=\dot{G}\cap D,
\] 
we get that $p\geq r$ is also in $H$ and thus that
$H\cap D\neq\emptyset$.

\item  \textbf{$M_1\cap\lambda\in\lambda>|M_1|=\nu$:}

First of all $M_1$ has size $|M_0|=\nu$ since it is the surjective image of 
$M_0\cap V^P$ and contains $M_0$.
Thus $\sup(M_1\cap\lambda)<\lambda$.
Now pick $\beta\in M_1\cap\lambda$. 
Find $\tau\in M_0$ such that $\Vdash_P\tau\in\lambda$ and for some
$q\in H$, $q\Vdash_P \tau=\beta$.  Let $\phi_\tau\in V^P\cap M_0$ be a $P$-name such that
$\Vdash_P\phi_\tau:\nu\to\tau$\emph{ is a bijection which belongs to $V$}.
Find $r\leq q$ $r\in H$ such that $r\Vdash\phi_\tau=\phi$ for some $\phi\in V$ 
bijection of $\nu$ with $\beta$. Since $\nu\subset M_0\subset M_1$ we get that
$\phi[\nu]=\beta\subset M_1$.
\end{itemize} 
The Claim and thus the Lemma are proved 
(Notice that the unique part of the proof in which we used that
 $\lambda$ is a successor cardinal is to get that $M_1\cap\lambda\in\lambda$).
\end{proof}
\end{proof}

\subsection{Preliminaries \textsf{III}: stationary sets and 
the stationary tower forcing}\label{sec:stattow}

$S$ is stationary if for all $f:P_\omega(\cup S)\to(\cup S)$ 
there is an $X\in S$ such that
$f[P_{\omega}(X)]\subset X$.

For a stationary set $S$ and a set $X$, if $\cup S\subseteq X$
we let $S^{X}=\{M\in P(X):M\cap\cup S\in S\}$,
if $\cup S\supseteq X$
we let $S\restriction X=\{M\cap X:M\in S\}$.

If $S$ and $T$ are stationary sets we say that $S$ and $T$ are compatible if
\[
S^{(\bigcup S)\cup(\bigcup T)}\cap T^{(\bigcup S)\cup(\bigcup T)}
\] 
is stationary.

We let $S\wedge T$ denote the set of $X\in P(\cup S\cup\cup T)$ such that
$X\cap\cup S\in S$ and $X\cap\cup T\in T$ and for all $\eta$
$\bigwedge \{S_\alpha:\alpha<\eta\}$ is the set of $M\in P(\bigcup_{\alpha<\eta}S_\alpha)$
such that $M\cap \cup S_\alpha\in S_\alpha$ for all $\alpha\in M\cap\eta$.

For a set $M$ we let $\pi_M:M\to V$ denote the transitive collapse of the structure $(M,\in)$ onto a transitive set $\pi_M[M]$ and we let $j_M=\pi_M^{-1}$.

For any regular cardinal $\lambda$
\[
R_\lambda=\{X:X\cap\lambda\in\lambda\text{ and }|X|<\lambda\}.
\]
and $\mathbb{R}^\lambda_\delta$ is the 
stationary tower whose elements are stationary sets $S\in V_\delta$ such that 
$S\subset R_\lambda$ with order given by $S\leq T$ if, letting $X=\cup(T)\cup\cup(S)$, 
$S^X$ is contained in $T^X$ modulo a club. 

$\mathbb{R}_\delta$ will denote $\mathbb{R}^{\aleph_2}_\delta$.

We recall that if $G$ is $V$-generic for $\mathbb{R}^\lambda_\delta$, then 
$G$ induces in a natural way a direct limit ultrapower embedding 
$j_G:V\to M^G$ where $[f]_G\in M^G$ if $f:P(X_f)\to V$ in $V$ and
$[f]_G\mathrel{R}_G [h]_G$ iff for some $\alpha<\delta$ such that 
$X_f,X_h\in V_\alpha$
we have that
\[
\{M\prec V_\alpha: f(M\cap X_f)\mathrel{R} h(M\cap X_h)\}\in G.
\]
If $M^G$ is well founded it is customary to identify $M^G$ with its transitive collapse.

We recall the following results about the stationary tower (see~\cite[Chapter 2]{LAR04}):

\begin{theorem}[Woodin]\label{thm:wstf}
Assume $\delta$ is a Woodin cardinal, $\lambda<\delta$ is regular 
and $G$ is $V$-generic for $\mathbb{R}^\lambda_\delta$. Then
\begin{enumerate}
\item \label{thm:wstf-1}
$M^G$ is a definable class in $V[G]$ and
\[
V[G]\models (M^G)^{<\delta}\subseteq M^G.
\]
\item\label{thm:wstf-2}
$V_\delta, G\subseteq M^G$ and $j_G(\lambda)=\delta$.
\item\label{thm:wstf-3}
$M^G\models\phi([f_1]_G,\dots,[f_n]_G)$ if and only for some $\alpha<\delta$ such that
$f_i:P(X_i)\to V$ are such that $X_i\in V_\alpha$ for all $i\leq n$:
\[
\{M\prec V_\alpha: V\models \phi(f_1(M\cap X_1),\dots,f_n(M\cap X_n))\}\in G.
\]
\end{enumerate}
Moreover by~\ref{thm:wstf-1} $M^G$ is well founded and thus can be identified with its transitive 
collapse. 
With this identifications we have that for every $\alpha<\delta$ 
and every set $X\in V_\alpha$,
$X=[\langle\pi_M(X): M\prec V_\alpha, X\in M\rangle]_G$.
In particular with this identification we get that
\[
(H_{j_G(\lambda)})^{M[G]}=V_\delta[G]=(H_\delta)^{V[G]}.
\]
and that $j_G\restriction H_\lambda^V$ is the identity and witnesses
that $H_\lambda^V\prec H_{j_G(\lambda)}^{V[G]}$.
\end{theorem}

\subsection{Preliminaries \textsf{IV}: Woodin cardinals are forcing axioms}\label{subsec:prel3}

The following is an outcome of Woodin's work on the stationary tower~\cite[Theorem 2.53]{woodin}.
\begin{lemma}[Woodin]\label{thm:wooFA}
Assume there are class many Woodin cardinals.
and $\lambda$ is a regular cardinal.
Then the following are equivalent:
\begin{enumerate}
\item\label{thm:wooFA-2}
 $S^\lambda_{\mathbb{P}}$ is stationary, where
 \[
S^\lambda_{\mathbb{P}}=\{M\prec H_{|\mathbb{P}|^+}:M\in R_\lambda\text{ and there is an 
$M$-generic filter for $\mathbb{P}$}\} 
 \] 
\item\label{thm:wooFA-3}
$\mathbb{P}$ completely embeds into $\mathbb{R}^\lambda_\delta\restriction T$
for some Woodin cardinal $\delta$ and some stationary $T\in\mathbb{R}^\lambda_\delta$. 
\end{enumerate}
\end{lemma}
\begin{proof}
The proof of this Lemma 
can be worked out along the same lines of
 the proof of Theorem~\ref{thm:MM++} in the next subsection. Thus we refer the reader to
 that proof.
\end{proof}

By Woodin's equivalence above and Lemma~\ref{lem:WOOFA}
we get the following:
\begin{theorem}{Woodin~\cite[Theorem 2.53]{woodin}}\label{thm:wooFA-bis}

Assume $V$ is a model of 
$\ZFC+$\emph{ there are class many Woodin cardinals},
and $\lambda=\nu^+$ is a successor cardinal in $V$.

Then 
the following are equivalent for any partial order $P\in V$:
\begin{enumerate}
\item
$S^\lambda_P$ is stationary.
\item
$\FA_{\nu}(P)$ holds.
\item
There is a locally complete
embedding of $P$ into $\mathbb{R}^\lambda_\delta$ for some Woodin cardinal $\delta>|P|$.
\end{enumerate}
\end{theorem}

$\mathsf{SSP}$ denote the class of posets which preserve stationary subsets of $\omega_1$.
Martin's maximum $\MM$ asserts that  $\FA_{\aleph_1}(P)$ holds for all $P\in\SSP$.

The following sums up the current state of affair regarding the classes $\Gamma_\lambda$
for $\lambda\leq\aleph_2$.

\begin{theorem}\label{thm:forcingaxioms-stateaffairs}
Assume there are class many Woodin cardinals. 
Then:
\begin{enumerate}
\item\label{thm:forcingaxioms-stateaffairs-1}
$\Gamma_{\aleph_1}$ is the class of all posets which regularly
embeds into some
$\mathbb{R}^{\aleph_1}_\delta$.
\item\label{thm:forcingaxioms-stateaffairs-2}
$\mathbb{R}^{\aleph_2}_\delta\in\SSP$ for any Woodin cardinal $\delta$.
\item\label{thm:forcingaxioms-stateaffairs-3}
$\MM$ holds if and only if
$\mathsf{SSP}$ is the class of all posets which regularly embeds
into $\mathbb{R}^{\aleph_2}_\delta$ for some Woodin cardinal $\delta$.
(Foreman, Magidor, Shelah~\cite{foreman_magidor_shelah}).
\end{enumerate}
\end{theorem}

\begin{proof}
We sketch a proof. 
\begin{description}
\item[\ref{thm:forcingaxioms-stateaffairs-1}] 

Trivial by Theorem~\ref{thm:wooFA-bis}.

\item[\ref{thm:forcingaxioms-stateaffairs-2}]

Let $S\in V$ be a stationary subset of $\omega_1$,
$G$ be $V$-generic for $\mathbb{R}^{\aleph_2}_\delta$ and $\dot{C}$ be a $\mathbb{R}^{\aleph_2}_\delta$-name for a club
subset of $\omega_1$. Then $\sigma_G(\dot{C})\in (H_{\omega_2})^{V[G]}=V_\delta[G]=(H_{\omega_2})^{M^G}$.
In particular there is some $f:P(V_\alpha)\to P(\omega_1)$ in $V_\delta$ such that
$[f]_G=\sigma_G(\dot{C})$. By Theorem~\ref{thm:wstf}.\ref{thm:wstf-3} the set of $M\prec V_\alpha$ such that
$f(M)$ is a club subset of $\omega_1$ in $V$ belongs to $G$.
Thus $f(M)\cap S$ is non empty for all such $M$, in particular
$M^G\models[f]_G\cap j_G(S)\neq\emptyset$. Now, since $j_G(\omega_1)=\omega_1$, we have that
$j_G(S)=S$. The conclusion follows.

\item[\ref{thm:forcingaxioms-stateaffairs-3}]
$\aleph_2$ is a a successor cardinal. 
For this reason, if $\MM$ holds, 
we can use the equivalence given by Theorem~\ref{thm:wooFA-bis} 
to get that any $P\in\SSP$
regularly embeds into some $\mathbb{R}^{\aleph_2}_\delta$. 
We can then use~\ref{thm:forcingaxioms-stateaffairs-2} 
to argue that if $P$ regularly embeds into some $\mathbb{R}^{\aleph_2}_\delta$ 
with $\delta$ a Woodin cardinal, 
then $P\in\SSP$.
\end{description}
\end{proof}

\subsection{Preliminaries \textsf{V}: $\MM^{++}$}\label{subsec:prel4} 
The ordinary proof of $\MM$ actually gives more information than what is captured by 
Theorem~\ref{thm:forcingaxioms-stateaffairs}.\ref{thm:forcingaxioms-stateaffairs-3}:
the latter asserts that any stationary set preserving poset $\mathbb{P}$ can be completely embedded into 
$\mathbb{R}^{\aleph_2}_\delta\restriction S^{\aleph_2}_{\mathbb{P}}$ for any Woodin cardinal 
$\delta>|\mathbb{P}|$ via some complete embedding $i$. However $\MM$ doesn't give much 
information on the nature of the 
complete embedding $i$.
On the other hand the standard model of $\MM$ provided by Foreman, Shelah and Magidor's 
consistency proof actually show that for any stationary set preserving
poset $\mathbb{P}$ and any Woodin cardinal $\delta>|\mathbb{P}|$
we can get a complete embedding $i:\mathbb{P}\to\mathbb{B}(\mathbb{R}^{\aleph_2}_\delta\restriction T)$
with a "nice" quotient forcing
$(\mathbb{R}^{\aleph_2}_\delta\restriction T)/i[\mathbb{P}]$.
For this reason we introduce the following well known variation of Martin's maximum:

\begin{definition}
$\MM^{++}$ holds if  
$T_{\mathbb{P}}$ is stationary for all $\mathbb{P}\in\SSP$, where $M\in T_{\mathbb{P}}$ iff
\begin{itemize}
\item
$M\prec H_{|\mathbb{P}|^+}$ is in $R_{\aleph_2}$,
\item
there is an $M$-generic filter $H$ for
$\mathbb{P}$ such that, if $G=\pi_M[H]$, $Q=\pi_M(\mathbb{P})$ and $N=\pi_M[M]$, 
then $\sigma_G:N^Q\to N[G]$ is an evaluation map such that
$\sigma_G(\pi_M(\dot{S}))$ is stationary  for all
$\dot{S}\in M$ $\mathbb{P}$-name for a stationary subset of $\omega_1$.
\end{itemize}
\end{definition}

The following is a well-known by-product of the ordinary consistency proofs of $\MM$ 
which to my knowledge is seldom explicitly stated:

\begin{theorem}[Foreman, Magidor, Shelah]
Assume $\kappa$ is supercompact in $V$, $f:\kappa\to V_\kappa$ is a Laver function and
\[
\{(P_\alpha,\dot{Q}_\alpha):\alpha\leq\kappa\}
\] 
is a revised countable support iteration such that
\begin{itemize}
\item
$P_\alpha\Vdash\dot{Q}_\alpha$ is semiproper, 
\item
$P_{\alpha+1}\Vdash|P_\alpha|=\aleph_1$,
\item
$\dot{Q}_\alpha=f(\alpha)$ if $P_\alpha\Vdash f(\alpha)$ is semiproper. 
\end{itemize}
Let $G$ be $V$-generic for $P_\kappa$. 
Then $\MM^{++}$ holds in $V[G]$.
\end{theorem}

\begin{theorem}\label{thm:MM++}
Assume there are class many Woodin cardinals.
Then the following are equivalent:
\begin{enumerate}
\item\label{thm:MM++-1}
$\MM^{++}$ holds.
\item\label{thm:MM++-2}
For every Woodin cardinal $\delta$ and every stationary set preserving
poset $\mathbb{P}\in V_\delta$ there is a complete embedding
$i:\mathbb{P}\to\mathbb{B}$  where 
$\mathbb{B}=\mathbb{B}(\mathbb{R}^{\aleph_2}_\delta\restriction T)$ for some
stationary set $T\in V_\delta$ such
that
\[
\Vdash_{\mathbb{P}}\mathbb{B}/i[\mathbb{P}]\text{ is stationary set preserving.}
\]
\end{enumerate}
\end{theorem}

The rest of this section is devoted to the proof of the above theorem.

\begin{proof}
We prove both implications as follows:
\begin{description}
\item[\ref{thm:MM++-1} implies~\ref{thm:MM++-2}]
We will show that if $G$ is $V$-generic for $\mathbb{R}_\delta$ with $T_{\mathbb{P}}\in G$
there is in $V[G]$ a $V$-generic filter $H$ for $\mathbb{P}$ such that
$\sigma_H(\dot{S})$ is stationary in $V[G]$ for all $\mathbb{P}$-names $\dot{S}$ for 
stationary subsets of 
$\omega_1$. Assume this is the case and let $\dot{H}$ be a $\mathbb{R}_\delta$-name for $H$ such that
$T_{\mathbb{P}}$ force the above property of $\dot{H}$ and $\mathbb{B}=\mathbb{B}(\mathbb{R}^{\aleph_2}_\delta\restriction T_{\mathbb{P}})$.
Then it is easy to check that the map
\begin{align*}
i & :\mathbb{P}\to\mathbb{B}
\\
& p\mapsto\|p\in\dot{H}\|_{\mathbb{B}}
\end{align*}
is a complete embedding such that 
\[
\Vdash_{\mathbb{P}}\mathbb{B}/i[\mathbb{P}]\text{ is stationary set preserving}.
\]
To define $\dot{H}$ we proceed as follows:
for each $M\in T_{\mathbb{P}}$ let $H_M\in V$ be $\pi_M[M]$-generic for 
$\pi_M(\mathbb{P})$ and such that $\sigma_{H_M}(\pi_M(\dot{S}))=S_M\in V$ is a stationary
subset of $\omega_1$ for all $\mathbb{P}$-name $\dot{S}\in M$ for a stationary subset of $\omega_1$.
Thus $[\langle H_M: M\in T_{\mathbb{P}}\rangle]_G$ is $V$-generic for $\mathbb{P}$.
Let $\dot{C}$ be a $\mathbb{R}_\delta$-name for a club subset of
$\omega_1$. As in the proof that $\mathbb{R}_\delta$ is stationary set preserving we can argue that
$\sigma_G(\dot{C})=[\langle C_M:M\prec V_\alpha\rangle]_G\in M^G$ is such that
$C_M\in V$ is a club subset of $\omega_1$  for some 
$\alpha<\delta$ and for all $M\prec V_\alpha$. Then 
\[
\sigma_G(\dot{C})\cap\sigma_H(\dot{S})=
[\langle C_M\cap S_M:M\in T_{\mathbb{P}}^{V_{\alpha}}\rangle]_G\neq \emptyset
\]
This shows that $[\langle H_M:M\in T_{\mathbb{P}}\rangle]_{\dot{G}}$ is the desired $\mathbb{R}_\delta$-name
$\dot{H}$, given that $\dot{G}$ is the canonical $\mathbb{R}_\delta$-name for a $V$-generic filter for
$\mathbb{R}_\delta$.

\item
\textbf{\ref{thm:MM++-2} implies~\ref{thm:MM++-1}.}
 
Let $\dot{G}$ be the canonical $\mathbb{R}_\delta$-name for a $V$-generic filter for
$\mathbb{R}_\delta$.
Let $T\in\mathbb{R}_\delta$ be a condition such that
$\mathbb{P}$ completely embeds into $\mathbb{B}=\mathbb{B}(\mathbb{R}_\delta\restriction T)$ via $i$ and
\begin{align*}
\Vdash_{\mathbb{P}}\mathbb{B}/i[\mathbb{P}]\text{ is stationary set preserving}
\end{align*}
 
Let $\dot{G}$ be the canonical name for the $\mathbb{R}_\delta$-generic filter and $\dot{H}=i^{-1}[\dot{G}]$.
Then $i(p)=\|p\in\dot{H}\|_{\mathbb{B}}$.

Now notice that $\mathbb{P}\in V_\delta$ and each 
$\mathbb{P}$-name $\dot{S}$ for a stationary subset of 
$\omega_1$ is in $V_\delta$ since it is given by $\omega_1$-many maximal antichains
of $\mathbb{P}$. 

Thus if $G$ is $V$-generic for $\mathbb{R}_\delta$ with $T\in G$,
$H=\sigma_G(\dot{H})\in V_\delta[G]$ is $V$-generic for $\mathbb{P}$ and is such that
$\sigma_{H}(\dot{S})\in V_{\delta}[G]$ is stationary in $V[G]$ for all names $\dot{S}\in V^{\mathbb{P}}$ for stationary subsets of $\omega_1$.
Since $V_\delta[G]=(H_{\omega_2})^{M^G}$, $H\in M^G$, 
so $H=[f]_G$ for some $f:P(V_\alpha)\to P(\mathbb{P})$. 
It is possible to check that for some $\alpha<\delta$
\begin{align*}
S= &\{M\prec V_\alpha: f(M)=H_M\text{ is a $\pi_M[M]$-generic filter for 
$\pi_M(\mathbb{P})$ such that}
\\
& \sigma_{H_M}(\pi_M(\dot{S}))\text{ is stationary for all names $\dot{S}\in V^{\mathbb{P}}\cap M$} 
\\
& \text{for stationary subsets of $\omega_1$}\}\in G
\end{align*}
In particular $S\leq T_{\mathbb{P}}$ is stationary and we are done.
\end{description}
\end{proof}

\section{Absoluteness of the theory of $H_{\aleph_2}$ in models of $\MM^{++}$}
\label{sec:mainth}

In this section we prove Theorem~\ref{thm:absMMLORD-0}. We leave to the reader to convert it into a proof of Theorem~\ref{thm:absMMLORD-1}


\begin{theorem}\label{thm:absMMLORD}
Assume $\MM^{++}$ holds in $V$ and there are class many Woodin cardinals.
Then the $\Pi_2$-theory of $H_{\aleph_2}$ with parameters cannot be changed by 
stationary set preserving forcings which preserve $\BMM$.
\end{theorem}

\begin{proof}
Assume $V$ models $\MM^{++}$ and let $P\in M$ be such that
$V^P$ models $\BMM$.

Let $\delta$ be a Woodin cardinal larger than $|P|$.
By Theorem~\ref{thm:MM++} there is a complete embedding 
$i:P\to Q=\mathbb{R}_\delta\restriction T_P$ for some
stationary set $T_P\in V_\delta$ such
that
\[
\Vdash_{P} Q/i[P]\text{ is stationary set preserving}.
\]

Now let $G$ be $V$-generic for $Q$ and $H=i^{-1}[G]$ be $V$ generic for $P$.
Then $V\subset V[H]\subset V[G]$ and $V[G]$ is a 
generic extension of $V[H]$ by a forcing which is
stationary set preserving in $V[H]$.
Moreover by Woodin's theorem on stationary tower forcing~\ref{thm:wstf}, 
we have that $H_{\aleph_2}^V\prec H_{\aleph_2}^{V[G]}$.

We show that 
\[
H_{\aleph_2}^V\prec_{\Sigma_2} H_{\aleph_2}^{V[H]}.
\]
This will prove the Theorem, modulo standard forcing arguments.

We have to prove the following for any $\Sigma_0$-formula $\phi(x,y,z)$:
\begin{enumerate}
\item \label{prf:1}
If 
\[
H_{\aleph_2}^V\models \exists y\forall x\phi(x,y,p)
\]
for some 
$p\in H_{\aleph_2}^V$, then
also 
\[
H_{\aleph_2}^{V[H]}\models\exists y\forall x\phi(x,y,p).
\]
\item \label{prf:2}
If 
\[
H_{\aleph_2}^V\models \forall y\exists x\phi(x,y,p)
\]
for some 
$p\in H_{\aleph_2}^V$, then
also 
\[
H_{\aleph_2}^{V[H]}\models\forall y\exists x\phi(x,y,p).
\]
\end{enumerate}

To prove~\ref{prf:1} we note that
for some $q\in H_{\aleph_2}^V$ we have that 
\[
H_{\aleph_2}^V\models\forall x\phi(x,q,p).
\]
Then, since 
\[
H_{\aleph_2}^V\prec H_{\aleph_2}^{V[G]},
\]
we have that 
\[
H_{\aleph_2}^{V[G]}\models \forall x\phi(x,q,p).
\] 
In particular, since $q,p\in H_{\aleph_2}^{V[H]}$ and
$H_{\aleph_2}^{V[H]}$ is a transitive substructure of $H_{\aleph_2}^{V[G]}$,
we get that 
\[
H_{\aleph_2}^{V[H]}\models \forall x\phi(x,q,p)
\] 
as well. 
The conclusion now follows.

To prove~\ref{prf:2} we note that, since 
\[
H_{\aleph_2}^V\prec H_{\aleph_2}^{V[G]},
\]
we have that 
\[
H_{\aleph_2}^{V[G]}\models \forall y\exists x\phi(x,y,p). 
\]
In particular we have that for any $q\in H_{\aleph_2}^{V[H]}$ 
we have that 
\[
H_{\aleph_2}^{V[G]}\models  
\exists x\phi(x,q,p). 
\]
Now, since $V[H]$ models $\BMM$ and $V[G]$ is an
extension of $V[H]$ by a stationary set preserving forcing, we get that
\[
H_{\aleph_2}^{V[H]}\prec_{\Sigma_1}H_{\aleph_2}^{V[G]}.
\]
In particular we can conclude that 
\[
H_{\aleph_2}^{V[H]}\models\exists x\phi(x,q,p)
\] 
for all $q\in H_{\aleph_2}^{V[H]}$, from which the desired conclusion follows.

The proof of the theorem is completed.
\end{proof}

\section{Woodin's absoluteness results for $H_{\aleph_1}$}\label{sec:WooHC}

Motivated by the results of the previous section we prove the following theorem:

\begin{theorem}\label{thm:WooHC}
Assume there are class many Woodin cardinals. Then the theory of 
$H_{\aleph_1}$ is invariant with respect to set forcing.
\end{theorem}

\begin{proof}

We prove by induction on $n$ the following Lemma, of which the Theorem is an 
immediate consequence:

\begin{lemma}\label{lem:ABSlemR}
Assume $V$ is a model of $\ZFC$ in which
there are class many Woodin cardinals. 
Let $P\in V$ be a forcing notion. 

Then for all $n$, 
$H_{\aleph_1}^V\prec_{\Sigma_n}H_{\aleph_1}^{V^P}$.
\end{lemma}

\begin{proof}
By Cohen's absoluteness Lemma~\ref{Lem:CohAbs}, we already know that 
for all models $M$ of
$\ZFC$ and all forcing $P\in M$
\[
H_{\aleph_1}^M\prec_{\Sigma_1}H_{\aleph_1}^{M^P}.
\]
Now assume that for all models $M$ of
$\ZFC+$\emph{there are class many Woodin cardinals} 
and all $P\in M$
we have shown that
\[
H_{\aleph_1}^M\prec_{\Sigma_n}H_{\aleph_1}^{M^P}.
\]
First observe that $M^P$ is still a model of 
$\ZFC+$\emph{there are class many Woodin cardinals}. 
Now pick $V$ an arbitrary model of 
$\ZFC+$\emph{there are class many Woodin cardinals} 
and $P\in V$ a forcing notion.

Let $\delta\in V$ be a Woodin cardinal in $V$ such that
$P\in V_\delta$.

To simplify the argument we assume $V$ is transitive and there is a
$V$-generic filter $G$ for $\mathbb{R}^{\aleph_1}_\delta$ 
(we leave to the reader to remove these unnecessary assumptions).

Then, since $\FA_{\aleph_0}(P)$ holds in $V$ and 
$P\in V_\delta$, by 
Theorem~\ref{thm:forcingaxioms-stateaffairs}.\ref{thm:forcingaxioms-stateaffairs-1}
there is in $V$ a complete embedding $i:P\to \mathbb{R}^{\aleph_1}_\delta$.
Let $H=i^{-1}[G]$. Then by our inductive assumptions applied to $V$ 
(with respect to $V[H]$) and to $V[H]$ (with respect to $V[G]$) we have 
that:
\[
H_{\aleph_1}^V\prec_{\Sigma_n}H_{\aleph_1}^{V[H]}
\prec_{\Sigma_n}H_{\aleph_1}^{V[G]}.
\]
By Woodin's work on the stationary tower forcing we also know that 
\[
H_{\aleph_1}^V\prec H_{\aleph_1}^{V[G]}.
\]
Now we prove that 
\[
H_{\aleph_1}^V\prec_{\Sigma_{n+1}}H_{\aleph_1}^{V[H]}.
\]
Since this argument holds for any $V$, $P$ and $G$, 
the proof will be completed.

We have to prove the following for any $\Sigma_n$-formula $\phi(x,z)$
and any $\Pi_n$-formula $\psi(x,z)$:
\begin{enumerate}
\item \label{prf:1-bis} 
If 
\[
H_{\aleph_1}^V\models\forall x\phi(x,p)
\] 
for some 
$p\in \mathbb{R}^V$, then
also 
\[
H_{\aleph_1}^{V[H]}\models \forall x\phi(x,p).
\]
\item \label{prf:2-bis} 
If 
\[
H_{\aleph_1}^V\models \exists x\psi(x,p)
\] 
for some 
$p\in \mathbb{R}^V$, then
also 
\[
H_{\aleph_1}^{V[H]}\models \exists x\psi(x,p).
\]

\end{enumerate}

To prove~\ref{prf:1-bis} we note that, since $H_{\aleph_1}^V\prec H_{\aleph_1}^{V[G]}$,
we have that 
\[
H_{\aleph_1}^{V[G]}\models \forall x\phi(x,p). 
\]
In particular we have that for any $q\in H_{\aleph_1}^{V[H]}$ 
we have that $H_{\aleph_1}^{V[G]}$ models that 
$\phi(q,p)$. Now, since by inductive assumptions
\[
H_{\aleph_1}^{V[H]}\prec_{\Sigma_n} H_{\aleph_1}^{V[G]},
\]
we get that
\[
H_{\aleph_1}^{V[H]}\models\phi(q,p)
\] 
for all $q\in H_{\aleph_1}^{V[H]}$, from which the desired conclusion follows.

To prove~\ref{prf:2-bis} we note that
for some $q\in H_{\aleph_1}^V$ we have that 
\[
H_{\aleph_1}^V\models \psi(q,p).
\]
Then, since by inductive assumptions we have that
\[
H_{\aleph_1}^V\prec_{\Sigma_n} H_{\aleph_1}^{V[H]},
\]
we conclude that 
\[
H_{\aleph_1}^{V[H]}\models \psi(q,p).
\] 
The conclusion now follows.

The lemma is now completely proved.
\end{proof}

The Theorem is proved.
\end{proof}

\begin{remark}
Theorem~\ref{thm:WooHC} has a 
weaker conclusion than~\cite[Theorem 3.1.7]{LAR04} where
it is shown that
in the presence of class many inaccessible limits of Woodin cardinals, 
the first order theory of $L(P_{\omega_1}\Ord)$ is invariant
with respect to set forcing. However in Theorem~\ref{thm:WooHC}
we have slightly weakened the large cardinal hypothesis with respect to 
Woodin's~\cite[Theorem 3.1.7]{LAR04}.

We had to weaken the conclusion of Theorem~\ref{thm:WooHC} with respect
to~\cite[Theorem 3.1.7]{LAR04}
since we cannot replace $H_{\aleph_1}$ with $L(\mathbb{R})$ (or $L(P_{\omega_1}\Ord)$)
in the proof of the above Lemma.
The reason is that any element of $L(\mathbb{R})$ is defined by an arbitrarily large ordinal and a real
and the ordinal may be moved by $j_G$, where 
$j_G:V\to M^G$ is the ultarpower embedding
living in $V[G]$ and induced by $G$. 
In particular we have that $j_G\restriction H_{\omega_1}^V$ is the
identity and witnesses that  
\[
H_{\aleph_1}^V\prec H_{\aleph_1}^{V[G]},
\]
but $j_G$ may not witness that 
\[
L(\mathbb{R})^{V}\prec L(\mathbb{R})^{V[G]},
\]
which is what we would need in order to perform the type of argument we 
performed in the proof of the Lemma. 
\end{remark}

\section{Questions and open problems}\label{sec:comop}

\subsection{A conjecture on $\MM^{++}$ and $\Gamma_{\aleph_2}$-logic.}
We conjecture the following:
\begin{conjecture}
Assume $V$ is a model of $\MM^{++}+$\emph{large cardinals}.
Then for every $P\in V$ which preserves $\MM^{++}$
\[
H_{\aleph_2}^V\prec H_{\aleph_2}^{V^P},
\] 
\end{conjecture}
There is a major obstacle in performing the arguments of  
Lemma~\ref{lem:ABSlemR} in combination with the proof
of Theorem~\ref{thm:absMMLORD} to prove this conjecture.

Assume $H$ is $V$-generic for $P$ and $G$ is $V$-generic for
$\mathbb{R}^{\aleph_2}_\delta$ so that:
\begin{itemize}
\item
$V\models\MM^{++}$, 
\item
$V[H]\models\MM^{++}$,
\item
$V[G]$ is an extension of $V$ and of $V[H]$ by a stationary set preserving forcing,
\item
$H_{\aleph_2}^V\prec_{\Sigma_1} H_{\aleph_2}^{V[H]}
\prec_{\Sigma_1} H_{\aleph_2}^{V[G]}$,
\item
$H_{\aleph_2}^V
\prec H_{\aleph_2}^{V[G]}$,
\end{itemize}
From these data following the proof of Theorem~\ref{thm:absMMLORD}
we can infer
$H_{\aleph_2}^V\prec_{\Sigma_2} H_{\aleph_2}^{V[H]}$, but we cannot infer
$H_{\aleph_2}^{V[H]}\prec_{\Sigma_2} H_{\aleph_2}^{V[G]}$ (which is what allows us to 
perform the next step in Lemma~\ref{lem:ABSlemR}) 
because we cannot prove that $V[G]$ is a model of $\BMM$ (and we do not expect this to
be the case).

Thus some new idea is required to prove (or disprove) this conjecture.



\subsection{What is the relation between $\MM^{++}$ and axiom $(*)$?}

It is well known that Woodin's  $(*)$-axiom is not compatible with the existence
of a well order of $P(\omega_1)$ definable in $H_{\aleph_2}$ without 
parameters.
On the other hand Larson has shown that $\MM^{+\omega}$ is 
compatible with the existence of a well-order of
$P(\omega_1)$ definable in $H_{\omega_2}$ without parameters~\cite{larson-dwo}. 
If we inspect Larson's result, we see that Larson's well order is neither 
$\Pi_2$-definable nor
$\Sigma_2$-definable over $H_{\omega_2}$.
Thus Theorem~\ref{thm:absMMLORD-1}
does not prove that Larson's well-order
can be defined in all models of $\MM^{++}$. Whether $\MM^{++}$ can imply or deny
axiom $(*)$ is an interesting open problem. Larson's result already shows that
any version of $\MM^{+\alpha}$ strictly weaker than $\MM^{++}$ neither denies nor
implies axiom $(*)$.
However for our absoluteness results it seems to be crucial that the ground model satisfy
$\MM^{++}$.

\begin{question}
Does $\MM^{++}+$\emph{large cardinals} denies or implies Woodin's $(*)$-axiom?
\end{question}

\subsection{A conjecture on $\Gamma_{\aleph_3}$-logic}

The results of this paper suggest the following definitions.

Let $\Gamma$ be a class of partial orders defined
by some parameter $\lambda$ which is a regular cardinal
definable in some theory $T$ extending $\ZFC$.
Natural examples of such classes $\Gamma$ for $\lambda=\aleph_2$ are the family of
stationary set preserving posets, semiproper posets, proper posets, CCC posets\dots.
 
$\phi(\Gamma,\lambda)$ asserts that
\begin{quote} 
For all $\mathbb{P}\in \Gamma$ and
all Woodin cardinal $\delta>|\mathbb{P}|$ there is
a complete embedding $i:\mathbb{P}\to\mathbb{R}^\lambda_\delta\restriction S$
for some stationary set $S\in V_\delta$ such that
\[
\Vdash_{\mathbb{P}}(\mathbb{R}^\lambda_\delta\restriction S)/
i[\mathbb{P}]\in (\Gamma)^{V^{\mathbb{P}}}
\]
\end{quote} 
Notice that $\phi(\Gamma,\lambda)$ entails that $\Gamma=\Gamma_\lambda$ by Woodin's 
theorem~\ref{thm:wooFA}.

\begin{definition}\label{def:maxplam}
A definable class of posets $\Gamma$ is 
maximal for $\lambda$ with respect to the theory $T$ if
$T$ models the following:
\begin{enumerate}
\item\label{def:maxplam-1}
$\Gamma_\lambda\subseteq\Gamma$.
\item\label{def:maxplam-2}
$\Con(T)\rightarrow\Con(T+\phi(\Gamma,\lambda))$.
\item \label{def:maxplam-3}
$\mathbb{R}^\lambda_\delta\in \Gamma$ for all Woodin cardinals $\delta>\lambda$.
\item
If $i:\mathbb{P}\to\mathbb{Q}$ is a locally complete embedding and $\mathbb{Q}\in\Gamma$, 
then $\mathbb{P}\in\Gamma$ as well.
\item
If for some definable class $\Gamma'$,
$\Con(T+\Gamma'\setminus\Gamma\neq\emptyset)$ then
$\Gamma'=\Gamma_\lambda$ is not consistent with $T$.
\end{enumerate}
\end{definition}

\begin{remark}
Notice that if $\delta_1<\delta_2$ are Woodin cardinals, then
$\mathbb{R}^\lambda_{\delta_1}\restriction T$ completely embeds into
$\mathbb{R}^\lambda_{\delta_2}\restriction T$ for all 
regular cardinals $\lambda<\delta_1$ and all stationary sets $T\in P(R_\lambda)\cap V_{\delta_1}$
(see~\cite[Exercise 2.7.15]{LAR04} and ~\cite[Lemma 2.7.14, Lemma 2.7.16]{LAR04}). 
\end{remark}

\begin{remark}
The following holds:
\begin{enumerate}
\item
The class of all posets is maximal for $\aleph_1$ relative to
$\ZFC+$\emph{ there are class many Woodin cardinals}.
\item
$\SSP$ is maximal for $\aleph_2$ relative to
$\ZFC+$\emph{ there are class many supercompact cardinals}.
\end{enumerate}
\end{remark}



\begin{conjecture}
There is a class $\Gamma_{\aleph_3}$ which is maximal with respect to the theory
$\ZFC+\MM^{++}+$\emph{ large cardinals}. 
\end{conjecture}

Notice that if the above conjecture stands, one should expect to be able to prove 
the analogue of Theorem~\ref{thm:absMMLORD}
for $H_{\aleph_3}$.

If the above approach is successful at $\aleph_3$, is there a 
cardinal for which it cannot work? I.e.:
\begin{question}
What about maximal classes $\Gamma_\lambda$ for larger cardinals $\lambda$?
\end{question}

\subsection{What about the effects of $\MM^{++}$ on the theory of
$L(P_{\omega_2}\Ord)$?}
Can the methods presented in this paper be of some use in the study of 
$L(P_{\omega_2}\Ord)$ and not just of $H_{\aleph_2}$? 
\begin{question}
Can $\MM^{++}$+\emph{large cardinals} decide in $\SSP$-logic the theory of
$L(P_{\omega_2}\Ord)$?
\end{question}

While we can effectively compute many of the
consequences $\MM^{++}$ has on the theory of $H_{\aleph_2}$, this is not the case for
$L(P_{\omega_2}\Ord)$, for example: 
by~\cite[Remark 1.1.28]{LAR04} $\ZFC$ fails in $L(P_\kappa\Ord)$ 
for all cardinals $\kappa$ if there are $\kappa^+$-many measurable cardinals in $V$.
\begin{question}
Assume $\MM^{++}$ holds. What is the least ordinal $\lambda$ for which 
\[
H_\lambda^{L(P_{\omega_2}\Ord)}\not\models\ZFC?
\]
\end{question}
It is not hard to see that in models of $\MM^{++}$ the several examples of definable 
(with parameter in $H_{\omega_2}$) well-orders of $P(\omega_1)$ provided by
results of Aspero, Caicedo, Larson, Moore, Todor\v{c}evi\'c, Veli\v{c}kovi\'c, 
Woodin and others show that $\lambda$ is larger than
$\aleph_2$ and is at most the $\omega_2+1$-th measurable cardinal of $V$.

\subsection*{Acknowledgements}
I wish to thank Joan Bagaria, Ilijas Farah, Paul Larson, Menachem Magidor,
and Boban Veli\v{c}kovi\'c 
for many useful comments and discussions 
on the material presented in this paper.
The author acknowledges support from the 2009 PRIN grant ``Modelli e Insiemi''
and from the Kurt G\"odel Research Prize Fellowship 2010.

\bibliographystyle{amsplain}
\bibliography{APALMV-2}

\end{document}